\documentclass[12pt]{amsart}

\author{Matthew Just}
\author{Paul Pollack}
\address{Department of Mathematics\\University of Georgia\\ Athens, GA 30602}
\email{justmatt@uga.edu}
\email{pollack@uga.edu}
\title{Variations on a theme of Schinzel and W\'ojcik}
\usepackage{amsmath,amssymb,amsthm,geometry,mathscinet,url,mathrsfs,microtype}
\usepackage{enumitem} \setlist[enumerate]{label={\upshape(\roman*)}}
\usepackage{parskip}
\usepackage[utf8]{inputenc}
\newcommand{\genlegendre}[4]{%
  \genfrac{(}{)}{}{#1}{#3}{#4}%
  \if\relax\detokenize{#2}\relax\else_{\!#2}\fi
}
\newcommand{\leg}[3][]{\genlegendre{}{#1}{#2}{#3}}

\DeclareMathAlphabet{\curly}{U}{rsfs}{m}{n}

\newtheorem{thm}{Theorem}

\newtheorem{lem}[thm]{Lemma}

\theoremstyle{remark}

\newtheorem{rmk}[thm]{Remark}
\newtheorem{rmks}[thm]{Remarks}

\theoremstyle{definition}

\begin{document}
\renewcommand{\labelenumi}{(\roman{enumi})}
\def\Ll{\mathcal{L}}
\def\N{\mathrm{N}}
\def\Aa{\mathcal{A}}
\def\I{\mathcal{I}}
\def\Q{\mathbb{Q}}
\def\O{\mathscr{O}}
\newcommand\rad{\mathrm{rad}}
\def\Z{\mathbb{Z}}
\def\F{\mathbb{F}}
\def\R{\mathbb{R}}
\def\C{\mathbb{C}}
\def\Pp{\mathcal{P}}
\def\Ss{\mathcal{S}}
\def\T{\mathbb{T}}
\def\ord{\mathrm{ord}}
\newcommand\Li{\mathrm{Li}}

\begin{abstract} Schinzel and W\'ojcik have shown that if $\alpha, \beta$ are rational numbers not $0$ or $\pm 1$, then $\ord_p(\alpha)=\ord_p(\beta)$ for infinitely many primes $p$, where $\ord_p(\cdot)$ denotes the order in $\F_p^{\times}$. We begin by asking: When are there infinitely many primes $p$ with $\ord_p(\alpha) > \ord_p(\beta)$? We write down several families of pairs $\alpha,\beta$ for which we can prove this to be the case. In particular, we show this happens for ``100\%'' of pairs $A,2$, as $A$ runs through the positive integers. We end on a different note, proving a version of Schinzel and W\'{o}jcik's theorem for the integers of an imaginary quadratic field $K$: If $\alpha, \beta \in \O_K$ are nonzero and neither is a root of unity, then there are infinitely many maximal ideals $P$ of $\O_K$ for which $\ord_P(\alpha) = \ord_P(\beta)$.
\end{abstract}
\maketitle

\section{Introduction}
Let $\alpha, \beta$ be rational numbers, not $0$ or $\pm 1$. For all but finitely many primes $p$, both $\alpha$ and $\beta$ are $p$-adic units, and so it is sensible to talk about their multiplicative orders upon reduction mod $p$. Schinzel and W\'{o}jcik \cite{sw92}, extending unpublished investigations of J.S. Wilson, J.G. Thompson, and J.W.S. Cassels, proved that there are infinitely many primes $p$ for which $\ord_p(\alpha) = \ord_p(\beta)$. Equivalently (since $\F_p^{\times}$ is cyclic), $\alpha$ and $\beta$ generate the same subgroup of $\F_p^{\times}$ infinitely often.

It is an open problem to characterize the triples $\alpha, \beta, \gamma \in \Q^{\times} \setminus \{\pm 1\}$ for which $\ord_p(\alpha)=\ord_p(\beta)=\ord_p(\gamma)$ infinitely often. But in  a recent preprint, J\"{a}rviniemi presents such a characterization not just for triples, but for tuples of any fixed length, conditional on the Generalized Riemann Hypothesis \cite{j20}. (See \cite{PS09} for earlier GRH-conditional results, and \cite{wojcik96,fouad18} for related results conditional not on GRH but on Schinzel's ``Hypothesis H'' \cite{ss58}.) Sticking instead to pairs $\alpha,\beta$ but taking the problem in a different direction, various authors have investigated the distribution of $p$ for which $\ord_p(\alpha) \mid \ord_p(\beta)$ (see \cite{MS00} and \cite{MSS19}).
%

It is known that if $\alpha, \beta \in \Q^{\times}\setminus\{\pm 1\}$ and  $\ord_{p}(\alpha)=\ord_{p}(\beta)$ for all but finitely many primes $p$, then $\alpha = \beta$ or $\alpha=\beta^{-1}$ (see \cite{schinzel70} or \cite{CS97}). A natural complement to the theorem of Schinzel and W\'{o}jcik would be a characterization of those pairs $\alpha,\beta \in \Q^{\times}\setminus \{\pm 1\}$ for which
\begin{equation}\label{eq:lessthan} \ord_p(\alpha) > \ord_p(\beta) \qquad\text{for infinitely many primes $p$.}\end{equation}
Call the (ordered) pair $\alpha, \beta$ \textsf{order-dominant} if \eqref{eq:lessthan} holds.

Under GRH, we have a completely satisfactory classification of order-dominant pairs. Assume, as above, that $\alpha, \beta \in \Q^{\times}\setminus\{\pm 1\}$. Then $\alpha, \beta$ is order-dominant if and only if $\alpha$ is not a power of $\beta$.\footnote{The ``only if'' half is clear. For the ``if'' direction: When $\alpha, \beta$ are multiplicatively independent, J\"{a}rviniemi \cite[Theorem 1.4]{j20} proves (under GRH) that $\ord_p(\alpha)/\ord_p(\beta)$ can be made arbitrarily large, which certainly implies the order-dominance of $\alpha,\beta$. When $\alpha,\beta$ are multiplicatively dependent but $\alpha$ is not a power of $\beta$, the order-dominance of $\alpha,\beta$ follows (unconditionally) from an  elementary argument with Zsigmondy's theorem.} It seems difficult to obtain a result of comparable strength unconditionally. Our first three theorems describe partial progress. Each reports on certain families of integers $A,B$ for which we can prove the order-dominance of $A,B$ without any unproved hypothesis. We mostly (but not exclusively) restrict attention to positive integers $A,B$; this allows us to illustrate the basic methods while avoiding technical complications. As will become clear shortly, the limitations of our methods manifest already in this restricted situation; given these limitations, we have tried optimize the exposition for clarity rather than generality.

Below, $\leg{\cdot}{\cdot}$ denotes the Legendre--Jacobi--Kronecker symbol.

\begin{thm}\label{thm:dominant} \mbox{ }
\begin{enumerate}
\item Let $A, B$ be odd positive integers. Then $A,B$ is order-dominant if either
\[ \leg{-B(1-B)}{A} = -1 \quad \text{or}\quad \leg{1-B}{A} = -1.\]
\item The pair $2,B$ is order-dominant for every odd positive integer $B$.
\item The pair $A,2$ is order-dominant for every odd positive integer $A$ with $\leg{-1}{A}=-1$ or $\leg{-2}{A}=-1$, i.e., all odd positive $A \not\equiv 1\pmod{8}$.
\item If $A,B$ are coprime positive integers with $B > A^4$, then $-A,B$ is order dominant.
\end{enumerate}
\end{thm}

For example, it follows from Theorem \ref{thm:dominant} and its proof (see Remark \ref{rmk:rmk1}(ii)) that if $A$ and $B$ are any of $2,3,5$, or $7$, and $A\ne B$, then there are infinitely many primes $p$ with $\ord_p(A) > \ord_p(B)$.

When $(A,B) \in \{(2, 3), (3,2), (2,5), (5,2)\}$, Theorem \ref{thm:dominant} was implicitly proved by Banaszak in \cite{banaszak98} (see the proofs of Theorems 1 and 2 in \cite{banaszak98}), although his results were not stated this way. Our proofs are essentially the same as his for these cases.

Theorem \ref{thm:dominant}(iii) leaves untouched the pairs $A,2$ with $A\equiv 1\pmod{8}$. We can show that most such pairs are order-dominant. In fact, we have the following stronger result.

\begin{thm}\label{thm:almostall} The pair $A,2$ is order-dominant for \emph{almost all} positive integers $A$, meaning that the set of exceptional $A$ has asymptotic density $0$.
\end{thm}

\noindent (Note that Theorem \ref{thm:almostall}, unlike Theorem \ref{thm:dominant}(iii), allows $A$ to be even.) The proof of Theorem \ref{thm:almostall} begins by establishing an explicit (though slightly technical) sufficient condition for $A,2$ to be order-dominant, involving properties of Fermat numbers. The $A$ for which this condition fails, which we term \textsf{anti-elite numbers}, are then shown to be rare. See Remark \ref{rmk:rmk2} for the list of anti-elite $A$ up to 150.

The proofs of Theorems \ref{thm:dominant} and \ref{thm:almostall}, when they succeed, prove more than the order-dominance of $\alpha, \beta$. For all the pairs handled there, what is actually proved is that for infinitely many primes $p$, the ratio $\ord_p(\alpha)/\ord_p(\beta)$ is a positive even integer. Evenness stems from the fact that the primes $p$ we produce have $\alpha$ not a square modulo $p$, which we detect by quadratic reciprocity. One might hope to use higher reciprocity laws to generate further examples of order-dominant pairs. Our next theorem, whose proof depends on cubic reciprocity, is a modest step in this direction.

\begin{thm}\label{thm:cubicreciprocity} Let $A$ be an integer for which $3\nmid A$ and $A^2\not\equiv 1\pmod{9}$. For infinitely many primes $p$, the ratio $\ord_p(A)/\ord_p(-3)$ is an integer multiple of $3$. Thus, both $A,-3$ and $A,3$ are order dominant.
\end{thm}
\noindent (To see the claim about $A,3$, observe that $\ord_p(3)$ is at most twice $\ord_p(-3)$, and so at most two-thirds of $\ord_p(A)$.) Unfortunately, the proof of Theorem \ref{thm:cubicreciprocity} is not very amenable to generalization, although certain other pairs with $B=3\square$ (i.e., $3$ times a square) could be treated in a similar fashion. Analogously, the law of biquadratic reciprocity could be used to establish order-dominance of certain pairs $A,B$ with $B=-\square$.

One consequence of Theorem \ref{thm:cubicreciprocity} is that the pair $4,3$ is order-dominant. This could certainly not be proved by the methods of Theorem \ref{thm:dominant} or \ref{thm:almostall}, since $4$ is a square modulo every  $p$.

Theorems \ref{thm:dominant}, \ref{thm:almostall}, and \ref{thm:cubicreciprocity} (as well as their methods of proof) still leave us quite far from the GRH-conditional characterization of order dominant pairs. An interesting, difficult-seeming test case is the problem of proving that
\[ \ord_{p}(17) > \ord_{p}(2) \qquad\text{for infinitely many primes $p$}. \]
We hope that interested readers will take up this challenge!

Our final theorem is of a quite different nature. We prove the analogue of Schinzel and W\'{o}jcik's result for the integers of an imaginary quadratic field.

\begin{thm}\label{thm:quadratic} Let $K$ be an imaginary quadratic field with ring of integers $\O_K$. For nonzero $\alpha, \beta \in \O_K$, neither of which is a root of unity, there are infinitely many prime ideals $P$ of $\O_K$ for which $\alpha$ and $\beta$ generate the same subgroup of $(\O_K/P)^{\times}$.
\end{thm}

For example, $1+i$ and $2+i$ generate the same subgroup of $(\Z[i]/(\pi))^{\times}$ for infinitely many Gaussian primes $\pi$.

While the proof of Theorem \ref{thm:quadratic} follows the same basic strategy as \cite{sw92}, there are essential differences. It is important for us to have available auxiliary primes $\ell$ for which the $\ell$th power map, mod $\ell$, is induced by a nontrivial automorphism of $K$. In fact, we will use that all primes $\ell\equiv -1\pmod{\Delta}$ have this property, where $\Delta$ is the discriminant of $K$; this explains the requirement in the theorem that $K$ is imaginary.

It would be interesting to relax the restriction in Theorem \ref{thm:quadratic} that $\alpha$ and $\beta$ be integers of the field $K$. While our method of proof works for many pairs of nonintegral $\alpha,\beta \in K$, an elegant general statement does not seem forthcoming by these arguments.

\subsection*{Notation and conventions} Since $\ord_P(\cdot)$ is being used for the multiplicative order mod $P$, the $P$-adic valuation will be denoted $v_P(\cdot)$. We use $\lambda(\cdot)$ for Carmichael's function; that is, $\lambda(n)$ is the exponent of the multiplicative group mod $n$. We write $\langle g\rangle$ for the cyclic subgroup generated by a group element $g$.

We say that a statement about positive integers $n$ holds \textsf{whenever $n$ is sufficiently divisible} if there is a positive integer $K$ such that the statement holds for all $n$ divisible by $K!$. Note that if each of two statements holds whenever $n$ is sufficiently divisible, then their conjunction holds for all sufficiently divisible $n$. One should think of the requirement that $n$ be sufficiently divisible as analogous to the condition, in real analysis, that $\epsilon$ be sufficiently close to $0$. In fact, this is a bit more than an analogy: Asking that $n$ be sufficiently divisible amounts precisely to asking that $n$ be close enough to $0$ in $\hat{\Z}$, the profinite completion of the integers.

The requirement of sufficient divisibility will come up in the following way. We have a commutative ring $R$, an ideal $I$, and an element $A \in R$ that is invertible modulo $I$. Then $A^n \equiv 1\pmod{I}$ whenever $n$ is sufficiently divisible. Of course, it is simple enough here to say that the congruence holds whenever $n$ is divisible by $\#(R/I)^{\times}$. But later it will be convenient to suppress explicit mention of the required divisibility conditions.

\section{First examples of order-dominant pairs: Proof of Theorem \ref{thm:dominant}}\label{sec:proofdominant}
Suppose that $p$ is a prime with $\leg{A}{p}=-1$ and that $p$ divides $A^{n}-B$ for some even positive integer $n$. Since $B \equiv A^{n} \equiv (A^{n/2})^2\pmod{p}$, we see that
\begin{itemize} \item $B$ is in the subgroup generated by $A$ mod $p$, \emph{and}
\item $B$ is a square mod $p$.
\end{itemize} Since $A$ is not a square mod $p$, it cannot be that $A$ is in the subgroup generated by $B$ mod $p$. Hence, $\langle B\bmod{p}\rangle \subsetneq \langle A\bmod{p}\rangle$, and $\ord_p(A) > \ord_p(B)$. So to prove $A,B$ is order-dominant, it suffices to produce infinitely many primes $p$ of this kind.

Consider the situation where $A,B$ are odd and positive with $\leg{-B(1-B)}{A}=-1$. Then $A$ is coprime to both $B$ and $1-B$. We will locate primes $p$ with $\ord_p(A) > \ord_p(B)$ from among the prime divisors of
\[ \frac{A^{n}-B}{B-1}, \]
for suitably chosen positive integers $n$. Loosely speaking, what we show is that as $n$ gets more and more divisible, our procedure reveals larger and larger primes $p$ with $\ord_p(A) > \ord_p(B)$. (Precisely: As $n$ approaches $0$ in  $\hat{\mathbb{Z}}$, the discovered prime $p$ approaches $\infty$ in $\bar{\R}=\R\cup \{\pm\infty\}$.)

If $n$ is sufficiently divisible, then $\frac{A^{n}-B}{B-1} = \frac{A^{n}-1}{B-1} - 1 \in \Z^{+}$, and (since $\gcd(A,4(B-1))=1$) in fact
$\frac{A^{n}-B}{B-1} \equiv -1 \pmod{4}$. By quadratic reciprocity (for the Jacobi symbol) and the first supplementary law,
\begin{align*} \leg{A}{(A^n-B)/(B-1)} &= (-1)^{(A-1)/2} \leg{(A^n-B)/(B-1)}{A}\\
&= \leg{-1}{A} \leg{-B(B-1)}{A} = \leg{-B(1-B)}{A} = -1.
\end{align*}
Thus, we can choose $p$ dividing $\frac{A^{n}-B}{B-1}$ with $\leg{A}{p}=-1$. Assuming that $n$ is even (which holds whenever $n$ is sufficiently divisible), we are in the situation described in the first paragraph of this section, and so $\ord_p(A) > \ord_p(B)$.

It remains to see that infinitely many distinct $p$ arise in this construction. For that, it is enough to show that if $p$ is a fixed prime and $n$ is sufficiently divisible, then $p$ does not divide $\frac{A^n-B}{B-1}$. If $p$ divides $A$, then $p \nmid A^n-B$ for any $n$, and so $p\nmid \frac{A^n-B}{B-1}$. So suppose $p\nmid A$. If $n$ is sufficiently divisible, $A^n \equiv 1\pmod{p(B-1)}$ and so $\frac{A^n-B}{B-1} \equiv -1\pmod{p}$. Hence, $p\nmid\frac{A^n-B}{B-1}$.

Now suppose that $A,B$ are odd and positive with $\leg{1-B}{A}=-1$. Again, $A$ is coprime to $B-1$. We look at primes dividing expressions of the form
\[ \frac{B A^{n}-1}{B-1}. \]
If $n$ is sufficiently divisible, then
\[ \frac{B A^{n}-1}{B-1} \in \Z^{+}, \quad\text{with}\quad \frac{B A^{n}-1}{B-1} \equiv 1\pmod{4}.  \]
Moreover,
\[ \leg{A}{(BA^n-1)/(B-1)} = \leg{(BA^n-1)/(B-1)}{A} = \leg{1-B}{A} = -1. \]
Hence, there is a prime divisor $p$ of $(BA^n-1)/(B-1)$ with $\leg{A}{p}=-1$. Assuming $n$ even, $1/B \equiv A^n \equiv (A^{n/2})^2$ mod $p$, and so (reasoning as in the first paragraph of this section)  $\langle 1/B\bmod{p}\rangle \subsetneq \langle A\bmod{p}\rangle$. Hence, $\ord_p A > \ord_p (1/B) = \ord_p B$. That infinitely many distinct $p$ arise follows from the observation that for any fixed $p$ not dividing $A$, and all $n$ that are sufficiently divisible, $\frac{BA^n-1}{B-1} \equiv \frac{B-1}{B-1} \equiv 1\pmod{p}$.

We turn now to (ii). To handle pairs $2,B$ with $B$ odd and positive, we look at $p$ dividing
\[ \frac{4\cdot 2^{n} - B}{|4-B|}. \]
Whenever $n$ is sufficiently divisible,
\[ \frac{4\cdot 2^{n} - B}{|4-B|}\in \Z^{+}, \quad\text{and}\quad \frac{4\cdot 2^{n} - B}{|4-B|} \equiv \pm 3\pmod{8}.\]
Thus, $\leg{2}{(4\cdot 2^n-B)/|4-B|} =-1$. Choose $p$ dividing $\frac{4\cdot 2^{n} - B}{|4-B|}$ with $\leg{2}{p}=-1$. Then $B\equiv 2^{n+2} \equiv (2^{(n/2+1)})^2\pmod{p}$, and so $\langle B\bmod{p}\rangle \subsetneq \langle 2\bmod{p}\rangle$. Hence, $\ord_p(2) > \ord_p(B)$. Infinitely many distinct $p$ arise this way since, for each fixed odd prime $p$ and all $n$ that are sufficiently divisible, $\frac{4\cdot 2^{n} - B}{|4-B|} \equiv \frac{4-B}{|4-B|} \equiv \pm 1\pmod{p}$.

We breeze over the proof of (iii), concerning pairs $A,2$ with $\leg{-1}{A}=-1$, since the argument parallels the ones already described. This time one looks at primes dividing $2 A^{n}-1$, with $n$ sufficiently divisible. If $\leg{-1}{A}=1$ but $\leg{-2}{A}=-1$, one considers prime divisors of $A^n-2$, with $n$ sufficiently divisible. We leave the details to the reader.

Finally we treat (iv). Let $A,B$ be coprime integers larger than $1$ with $B> A^4$. We look at primes dividing
\[ \frac{A^{4+n}-B}{B-A^4}. \]
For each prime $p$,
\[ v_p\left(\frac{A^{4+n}-B}{B-A^4}+1\right)= v_p(A^n-1) + v_p(A^4)-v_p(B-A^4). \]
If $p$ is fixed and $n$ is sufficiently divisible, then the right-hand side is positive and in fact exceeds $v_p(4A)$: If $p\mid A$, this is clear, since $v_p(B-A^4)=0$ while $v_p(A^4)> v_p(4A)$. If $p\nmid A$, we use that $v_p(A^n-1)$ can be made arbitrarily large by making $n$ sufficiently divisible. It follows that $\frac{A^{4+n}-B}{B-A^4}$ is an integer for all sufficiently divisible $n$ and that
\[ \frac{A^{4+n}-B}{B-A^4} \equiv -1\pmod{4A}. \]
Hence, $\leg{-4A}{(A^{4+n}-B)/(B-A^4)} = \leg{-4A}{-1} =-1$. (We have $\leg{-4A}{-1}=-1$ since $-4A$ is an example of a negative discriminant; one reference for this is \cite[\S9.3]{MV07}.) Choose a prime $p$ dividing $\frac{A^{4+n}-B}{B-A^4}$ with $\leg{-4A}{p}=-1$. Since $p \mid (-A)^{4+n}-B$ and $-A$ is not a square mod $p$, a familiar argument shows that $\ord_p(-A) > \ord_p(B)$. Our above calculation with valuations implies that if $p$ is fixed, then $v_p(\frac{A^{4+n}-B}{B-A^4})=0$ for all sufficiently divisible $n$, and so this construction produces infinitely many different primes.

\begin{rmks}\label{rmk:rmk1}\mbox{ }
\begin{enumerate}
\item[(i)] A slight variant of the proof of Theorem \ref{thm:dominant}(iv) establishes the following more general result. Let $A, B$ be integers larger than $1$. Let $r_0$ be a nonnegative integer such that $v_p(A^{r_0}) \ge v_p(B)$ for all primes $p$ dividing $A$, and let $r$ be an even integer with $r> r_0+3$. If $B > A^r$, then $-A,B$ is order-dominant.

    Using this result, it is straightforward to show that for each fixed $A>1$, and almost all positive integers $B$ (in the sense of asymptotic density), the pair $-A,B$ is order-dominant.
\item[(ii)] The cases discussed in Theorem \ref{thm:dominant} were chosen as  representative of the basic method, but there are pairs of positive integers not covered by the conditions of Theorem \ref{thm:dominant} which can be shown order-dominant by this same strategy. One such pair is $3,7$ (look at primes dividing $\frac{7\cdot 3^n-1}{2}$), and another is $2,6$ (look at primes dividing $2^{n+1}-3$).
\end{enumerate}
\end{rmks}

\section{Almost all pairs $A,2$ are order-dominant: Proof of Theorem \ref{thm:almostall}}\label{sec:almostall}
The basic idea for the proof of Theorem \ref{thm:almostall} is encapsulated in the next lemma. Let $F_n = 2^{2^{n}}+1$ (for $n=0,1,2,3\dots$), the $n$th \textsf{Fermat number}. It is well-known that the $F_n$ are pairwise relatively prime and that if $p$ is a prime divisor of $F_n$, where $n\ge 2$, then $\ord_{p}(2) = 2^{n+1}$ and $2^{n+2} \mid p-1$ (see pages 5, 84 of \cite{ribenboim96}).

\begin{lem}\label{lem:NAE} Suppose $A$ is a positive integer with the property that
\[ \leg{A}{F_n} = -1 \quad\text{for infinitely many positive integers $n$}. \]
Then $A, 2$ is order-dominant.
\end{lem}

\begin{proof} Choose $n\ge 2$ with $\leg{A}{F_n}=-1$. There is a prime $p$ dividing $F_n$ with $\leg{A}{p}=-1$, and for this prime, $A^{(p-1)/2} \equiv -1\pmod{p}$. Hence, $\ord_p(A)$ divides $p-1$ but does not divide $\frac{p-1}{2}$, forcing $v_2(\ord_{p}(A)) = v_2(p-1)$. It follows that
\[ \ord_p(A) \ge 2^{v_2(p-1)} \ge 2^{n+2} > 2^{n+1} = \ord_p(2). \]
Since $p > \ord_p(A) \ge 2^{n+2}$, and $n$ can be chosen arbitrarily large, there are infinitely many $p$ with $\ord_p(A) > \ord_p(2)$.
\end{proof}

Primes $A$ \emph{failing} the hypothesis of Lemma \ref{lem:NAE} appear already in the literature;  M\"{u}ller \cite{muller07} calls these \textsf{anti-elite primes}. That is, $A$ is anti-elite if $\leg{A}{F_n}=1$ for all large enough positive integers $n$. We will call any integer $A$ satisfying this condition an \textsf{anti-elite integer}.

As M\"{u}ller observed, trivial changes to the proof of Theorem 4 in \cite{KLS02} show that anti-elite primes are sparse within the collection of all primes. Specifically, the count of anti-elite primes not exceeding $x$ is $O(x/(\log{x})^{3/2})$, for all $x\ge 2$.\footnote{A stronger upper bound of $O(x/(\log{x})^2)$ is claimed in \cite{KLS02}. Just \cite{just20} points out a small error in the proof and notes that, when corrected, $2$ must be replaced by $3/2$. In fact, one can recover an estimate almost as strong as originally claimed by a modification of the proof; see the end of our \S3.} In view of Lemma \ref{lem:NAE}, to prove Theorem \ref{thm:almostall} it is enough to show that only $o(x)$ positive integers $A\le x$ are anti-elite, as $x\to\infty$. We prove this in the following more precise form.

\begin{thm}\label{thm:aecount} For each $\epsilon > 0$ and all $x>x_0(\epsilon)$, the number of anti-elite $A \in (1,x]$ is $O_{\epsilon}(x/(\log{x})^{1-\epsilon})$. \end{thm}

\begin{proof} Write $A = A_0 A_1$, where $A_1$ is the largest odd divisor of $A$. We will assume that $v_2(\lambda(A_1)) < T-2$, where
\[ T := \left\lfloor \frac{\log(\log x/\log\log{x})}{\log{2}}\right\rfloor. \]
If $v_2(\lambda(A_1))\ge T-2$, then there is a prime $p$ dividing $A$ with $p \equiv 1\pmod{2^{T-2}}$, and the number of such $A\le x$ is
\[ \ll x \sum_{\substack{p \le x \\ p\equiv 1\pmod{2^{T-2}}}} \frac{1}{p} \ll x\frac{\log\log{x}}{2^{T-2}} \ll \frac{x (\log\log{x})^2}{\log{x}}, \]
which is $O_{\epsilon}(x/(\log{x})^{1-\epsilon})$. Here the sum on $p$ has been estimated by the Brun--Titchmarsh inequality \cite[Theorem 3.9, p.\ 90]{MV07} and partial summation.

We fix a nonnegative integer $t< T-2$ and count the number of anti-elite $A\in (1,x]$ with $v_2(\lambda(A_1))=t$. For each such $A$, the sequence $\{\leg{A}{F_n}\}_{n\ge t+2}$ is purely periodic. Indeed, if  $n\ge t+2$, then $n\ge 2$, so that $F_n \equiv 1\pmod{8}$ and $\leg{2}{F_n}=1$. Hence, $\leg{A}{F_n} = \leg{A_1}{F_n} = \leg{F_n}{A_1}$, which depends only on $F_n$ modulo $A_1$. In turn, $F_n = 2^{2^n} +1$ mod $A_1$ depends only on $2^n$ modulo $\lambda(A_1)$. Write
\[ \lambda(A_1) = 2^{v_2(\lambda(A_1))} B, \]
where $B$ is odd. Since $n\ge t= v_2(\lambda(A_1))$, the residue class of $2^n$ mod $\lambda(A_1)$ is determined by $2^n$ modulo $B$, which depends only on $n$ modulo $\lambda(B)$. Collecting our results, we see that $\{\leg{A}{F_n}\}_{n\ge t+2}$ is purely periodic (with period dividing $\lambda(B)$).

Since $A$ is anti-elite, it must be that each $F_n$ with $n\ge t+2$ satisfies $\leg{A}{F_n}=1$.  In particular,
\begin{equation}\label{eq:pfixed} \leg{A}{F_n}=1\quad\text{for all $n$ with}\quad t+2 \le n < T. \end{equation}
Factor $A=ps$, where $p$ is prime, $p\equiv 1\pmod{2^t}$. Our argument to bound the number of remaining $A$ assumes two different forms according to the sizes of $p$ and $s$.

Suppose first that $s\le \sqrt{x}$, so that $x/s\ge \sqrt{x}$. It follows from \eqref{eq:pfixed} that $p, s$ are prime to $\prod_{t+2 \le n < T} F_n$, and that
\begin{equation}\label{eq:sfixed}
 \leg{p}{F_n} =\leg{s}{F_n} \quad\text{whenever}\quad t+2 \le n < T.
\end{equation}
We view $s$ as fixed and count the number of corresponding $p$. Let $F = \prod_{t+2 \le n < T}F_n$. Keeping in mind that $p\equiv 1\pmod{2^t}$, we deduce from \eqref{eq:sfixed} that $p$ belongs to one of $\prod_{t+2 \le n < T} (\frac{1}{2}\phi(F_n)) = 2^{t-T+2} \phi(F)$ coprime residue classes modulo $2^t F$. (We use here that each symbol $\leg{\cdot}{F_n}$ is a nontrivial quadratic character mod $F_n$, since $F_n$ is not a square.) Notice that
\[ F < \prod_{n=0}^{T-1} F_n = F_{T}-2 < F_T. \]
So by our choice of $T$, and the inequality $t < T-2$, we have $2^t F < 2^t F_T = x^{o(1)}$ (as $x\to\infty$). Since $p =A/s \le x/s$,   the Brun--Titchmarsh inequality tells us that the number of possibilities for $p$ is $O(2^{-T} \frac{x}{s\log{x}})$. Summing on $s\le \sqrt{x}$ shows that the number of possible $A$ in this case is
\[ \ll \frac{x}{2^T} \ll \frac{x\log\log{x}}{\log{x}}. \]

Now suppose that $s > \sqrt{x}$. Then $p\le x/s < \sqrt{x}$. From \eqref{eq:pfixed}, we have with $m=A$ that
\[ \frac{1}{2^{T-t-2}} \prod_{t+2 \le n < T} \left(1+\leg{m}{F_n}\right) = 1. \]
Since the above left-hand side is nonnegative for every $m$, we conclude that an upper bound for the count of remaining $A$ is
\[ \frac{1}{2^{T-t-2}} \sum_{\substack{p\le \sqrt{x} \\ p \equiv 1\pmod{2^t}}}\sum_{s \le x/p} \prod_{t+2 \le  n< T} \left(1+\leg{sp}{F_n}\right).\]
Expanding the product and bringing the sums on $s,p$ inside gives a main term of size
\[ \frac{1}{2^{T-t-2}} \sum_{\substack{p\le \sqrt{x} \\ p \equiv 1\pmod{2^t}}}\sum_{s \le x/p} 1 \ll \frac{1}{2^{T-t}} x \sum_{\substack{p\le \sqrt{x} \\ p \equiv 1\pmod{2^t}}}\frac{1}{p} \ll \frac{x\log\log{x}}{2^T} \ll \frac{x(\log\log x)^2}{\log{x}}. \]
There are also $2^{T-t-2}-1$ error terms of the form $\frac{1}{2^{T-t-2}}\sum_{p,s} \leg{ps}{D}$, where $D$ is the product of some nonempty subset of $\{F_{t+2}, F_{t+3},\dots, F_{T-1}\}$.  Since Fermat numbers are pairwise coprime, $D$ is not a square, and $\leg{\cdot}{D}$ is a nontrivial Dirichlet character modulo $D$. Moreover, $D\le F = x^{o(1)}$. Using the trivial bound of $D$ for a nontrivial character sum mod $D$, we see that
\begin{align*} \sum_{\substack{p\le \sqrt{x} \\ p \equiv 1\pmod{2^t}}}\sum_{s \le x/p} \leg{ps}{D} &= \sum_{\substack{p\le \sqrt{x} \\ p \equiv 1\pmod{2^t}}}\leg{p}{D} \sum_{s \le x/p} \leg{s}{D} \\&\ll D \sum_{\substack{p \le\sqrt{x} \\ p \equiv 1\pmod{2^t}}}1 \ll D x^{1/2}. \end{align*} Hence, the errors contribute $\ll D x^{1/2} \ll x^{2/3}$. This is negligible compared to our main term, and so the number of $A$ that arise in this second case is $O(x(\log\log{x})^2/\log{x})$.

Assembling our results, we have proved that for each $t$, the number of corresponding $A$ is $O(x (\log\log{x})^2/\log{x})$. It remains to sum on $t$. But there are only $O(\log\log{x})$ possible values of $t$, and so the total number of anti-elite $A \le x$ is $O(x(\log\log{x})^3/\log{x})$, which is $O_{\epsilon}(x/(\log{x})^{1-\epsilon})$. \end{proof}

\begin{rmk}\label{rmk:rmk2} The anti-elite numbers up to $150$ are
\begin{multline*}1, \textbf{2}, 4, 8, 9, \textbf{13}, 15, 16,  \textbf{17}, 18, 21, 25, 26, 30, 32, 34, 35, 36, 42, 49, 50, 52, 60, \\64, 68, 70, 72, 81, 84, \textbf{97}, 98, 100, 104, 117, 120, 121, 123, 128, 135, 136, 140, 144. \end{multline*}
Anti-elite primes are shown in bold.
\end{rmk}

The proof of Theorem \ref{thm:aecount} is a more careful variant of the proof of Theorem 4 in \cite{KLS02}, the primary difference being that we keep track of the exact value of $t$ (the original argument only tracked whether $t$ was small or large, in a certain sense). Inserting this idea back into \cite{KLS02} will show that the count of elite primes up to $x$ is $O_\epsilon(x/(\log{x})^{2-\epsilon})$, essentially recovering the bound of $O(x/(\log{x})^2)$ claimed in \cite{KLS02}. Under GRH, the first author showed in \cite{just20} that the count of elite primes up to $x$ is $O_\epsilon(x^{5/6+\epsilon})$; the present method allows us to replace $5/6$ by $3/4$.

\section{Order-dominant pairs $A,-3$ and $A,3$: Proof of Theorem \ref{thm:cubicreciprocity}}
Let $\zeta = e^{2\pi i/3} = \frac{-1+\sqrt{-3}}{2}$. Below, we work in the ring $\Z[\zeta] = \O_K$, where $K=\Q(\zeta)$. Let $\lambda = 1-\zeta$, so that $\lambda^2 = -3\zeta$ and the ideal $(\lambda^2) = (3)$.

Take first the case when $A$ is even. Thinking of $n$ as sufficiently divisible (and in particular, even), we set
\[ \beta := A^{n/2} - \sqrt{-3} \]
and we attempt to evaluate the cubic residue symbol $\leg{A}{\beta}_3$. Since $\sqrt{-3} = 2\zeta+1$, we have
\begin{align} \beta &= A^{n/2}- 1 - 2\zeta\label{eq:repbeta}.\end{align} Since $3\nmid A$, for sufficiently divisible $n$ we find that $A^{n/2}\equiv 1\pmod{3}$, so that
\[ \beta \equiv -2\zeta \pmod {\lambda^2}. \]
Hence, $\zeta^2 \beta$ is congruent, modulo $\lambda^2$, to a rational integer coprime to $3$; that is, $\zeta^2\beta$ is \textsf{primary} in the sense required for an application of Eisenstein's $\ell$th power reciprocity law with $\ell=3$ (see, e.g., pp.\ 206--207 of \cite{IR90}). By that law, we deduce that (for sufficiently divisible $n$)
\[ \leg{A}{\beta}_3 = \leg{A}{\zeta^2 \beta}_3 = \leg{\zeta^2 \beta}{A}_3 = \leg{-\zeta^2 \sqrt{-3}}{A}, \]
so that
\[ \leg{A}{\beta}_3^2 = \leg{-3\zeta}{A}_3 = \leg{\lambda^2}{A}_3 = \leg{\lambda}{A}_3^2, \]
forcing $\leg{A}{\beta}_3 = \leg{\lambda}{A}_3$, since $\leg{A}{\beta}_3$ and $\leg{\lambda}{A}_3$ are third roots of unity. From the supplementary laws for Eisenstein reciprocity (see p.\ 365 of \cite{lemmermeyer00}),
\[ \leg{\lambda}{A}_3 = \leg{\zeta}{A}_3^{\frac{1}{2}(3-1)} = \leg{\zeta}{A}_3 = \zeta^{(A^2-1)/3}. \]
Since $A^2\not\equiv 1\pmod{9}$, the exponent on $\zeta$ is not a multiple of $3$. Thus, $\leg{A}{\beta}_3 \ne 1$. In particular, $A$ is not a cube modulo $\beta$, in $\Z[\zeta]$.

Since $A$ is even, we see from \eqref{eq:repbeta} that when $\beta$ is written as a $\Z$-linear combination of $1,\zeta$, the coefficient of $1$ and the coefficient of $\zeta$ are relatively prime. For any $\beta$ of this kind, a straightforward calculation shows that the canonical map $\Z \to \Z[\zeta]/(\beta)$ is surjective, and so induces an isomorphism $\Z/(N\beta) \cong \Z[\zeta]/(\beta)$. Thus, the calculation of the last paragraph implies that $A$ is not a cube modulo $N\beta = A^n+3$, in $\Z$. If $A$ were a cube modulo every prime factor of $A^n+3$, then $A$ would be a cube modulo $A^n+3$, by Hensel's lemma and the Chinese remainder theorem. (We use here that $A$ is prime to $A^n+3$, and that $3\nmid A^n+3$.) So we can choose a prime $p$ dividing $A^n+3$ with $A$ not a cube modulo $p$.

If $n$ is sufficiently divisible, then $3 \mid n$. Then $A^n \equiv -3\pmod{p}$ implies that $-3$ is a cube modulo $p$ and that $-3$ mod $p$ belongs to the subgroup generated by $A$ mod $p$. Since $A$ is not a cube mod $p$, we see $A$ is not in the subgroup generated by $-3$, and thus $\langle -3\bmod{p}\rangle \subsetneq \langle A\bmod{p}\rangle$. It follows that $\ord_p(A)/\ord_p(3)$ is a positive integer. To see that this integer is a multiple of $3$, notice that $p\equiv 1\pmod{3}$ (otherwise, $A$ would be a cube mod $p$), that $v_3(\ord_p(A)) = v_3(p-1)$ (since $A$ is a not a cube) and that $v_3(\ord_p(-3)) < v_3(p-1)$ (since $-3$ is a cube). Thus, $v_3(\ord_p(A)/\ord_p(3)) \ge 1$.

We have shown so far that if $n$ is sufficiently divisible, one can find a prime factor of $A^n+3$ with $\ord_p(A)/\ord_p(3)$ an integer multiple of $3$. To see that infinitely many distinct primes arise, notice that all of the $p$ produced by this construction are odd and coprime to $A$. Then observe that if $p$ is any fixed prime not dividing $2A$, then $A^n+3 \equiv 4 \not\equiv 0\pmod{p}$ whenever $n$ is sufficiently divisible.

The proof is essentially the same when $A$ is odd, except that now one should set $\beta := \frac{1}{2}(A^{n/2} - \sqrt{-3})$. It is also useful to observe that $\leg{2}{A}_3 = \leg{A}{2}_3 = \leg{1}{2}_3=1$. We leave the details to the reader.

\section{Equal orders in imaginary quadratic rings: Proof of Theorem \ref{thm:quadratic}}
Let $K$ be a quadratic field of discriminant $\Delta < 0$, and let $\alpha, \beta$ be distinct nonzero elements of $\O_K$, neither of which is a root of unity. Let $I$ be the largest ideal divisor of $(\beta-\alpha)$ coprime to $(\alpha)$. The prime ideals $P$ referred to in  the conclusion of Theorem \ref{thm:quadratic} will come to us as divisors of the (ideal) expression
\[  (\beta \alpha^\ell -1)/I, \]
where $\ell$ is a prime number for which $\ell+1$ is sufficiently divisible. It is important to note that any ``sufficiently divisible'' hypothesis on $\ell+1$ is always satisfied by infinitely many primes $\ell$; this follows, e.g., from Dirichlet's theorem on primes in progressions. (For an elementary proof of the $-1\bmod{M}$ case of Dirichlet's theorem used here, see \S50 of \cite{nagell51}.)

If $\ell+1$ is sufficiently divisible, then $\alpha^{\ell+1}\equiv 1\pmod{I}$, so that $ \alpha(\beta \alpha^{\ell}-1) \equiv \beta-\alpha\equiv 0 \pmod{I}$. Hence, $(\beta \alpha^\ell -1)/I$ is a nonzero, integral ideal of $\O_K$. Since $\Delta\mid \ell+1$ when $\ell+1$ is sufficiently divisible, $$\sqrt{\Delta}^{\ell} \equiv \Delta^{(\ell-1)/2} \sqrt{\Delta} \equiv \leg{\Delta}{\ell} \sqrt{\Delta} \equiv \leg{\Delta}{-1} \sqrt{\Delta} \equiv -\sqrt{\Delta} \pmod{\ell}.$$ So using a bar for complex conjugation (identified with the nontrivial automorphism of $K$), $\alpha^{\ell} \equiv \bar{\alpha} \pmod{\ell}$, and
\begin{align*} N((\beta \alpha^\ell -1)/I) &= N(\beta \alpha^{\ell}-1)/N(I) \\
&\equiv N(\beta \bar{\alpha} - 1)/N(I) \pmod \ell.\end{align*}
In the last line, division by $N(I)$ mod $\ell$ is to be understood as multiplication by the inverse of $N(I)$ mod $\ell$.  The rational number $N(\beta\bar{\alpha}-1)/N(I)$ exceeds $1$, since
\begin{align*}
   N(\beta\bar{\alpha}-1) - N(I) &\ge N(\beta \bar{\alpha} -1) - N(\beta-\alpha)
                               \\&= (\beta\bar{\alpha}-1)(\bar{\beta} \alpha-1) - (\beta-\alpha)(\bar{\beta} -\bar{\alpha})\\
                               &= (\beta\bar{\beta}-1)(\alpha\bar{\alpha}-1)= (N\alpha - 1) (N\beta-1) > 0.
\end{align*}
It follows that if $\ell+1$ is sufficiently divisible,
\[ N(\beta \bar{\alpha} - 1)/N(I) \not\equiv 1 \pmod \ell.\]
Thus, there must be a prime ideal $P$ of $\O_K$ dividing $(\beta \alpha^\ell -1)/I$ with $N(P) \not \equiv 1 \pmod{\ell}$, i.e., with $\ell \nmid (\O_K/P)^{\times}$. Since $\beta \alpha^{\ell}\equiv 1\mod{I}$, we deduce that $\langle \beta \bmod P\rangle = \langle \alpha^{-\ell} \bmod P\rangle = \langle \alpha \bmod P\rangle$.

To show that infinitely many such $P$ arise, we show that any fixed $P$ is coprime to $(\beta \alpha^{\ell} - 1)/I$ for all $\ell$ with $\ell+1$ sufficiently divisible. This is clear if $P \mid (\alpha)$. Otherwise, choose $k$ for which $P^k \parallel (\beta-\alpha)$. Then $P^k \parallel I$. Whenever $\ell+1$ is sufficiently divisible,
\[ \alpha(\beta \alpha^{\ell}-1) \equiv (\beta-\alpha) \pmod{P^{k+1}}, \]
which implies that $P^k \parallel (\beta \alpha^{\ell}-1)$. But then $P \nmid (\beta \alpha^{\ell}-1)/I$.

\section*{Acknowledgements}
The first author (M.J.) was supported by the UGA Algebraic Geometry, Algebra, and Number Theory RTG grant, NSF award DMS-1344994. The second author (P.P.) was supported by NSF award DMS-2001581. We thank Michael Filaseta, Pieter Moree,  Carl Pomerance, and Enrique Trevi\~no for helpful comments. We are also grateful to MathOverlow user \texttt{Hhhhhhhhhhh} for the post which brought this question to our attention \cite{mathoverflow20}.

\providecommand{\bysame}{\leavevmode\hbox to3em{\hrulefill}\thinspace}
\providecommand{\MR}{\relax\ifhmode\unskip\space\fi MR }
\providecommand{\MRhref}[2]{%
  \href{http://www.ams.org/mathscinet-getitem?mr=#1}{#2}
}
\providecommand{\href}[2]{#2}

\end{document}